\documentclass[a4paper, 12pt]{article}
\usepackage{graphicx}
\usepackage{amssymb,amsfonts,amsthm, amscd}
\usepackage{amsmath}
\usepackage{hyperref}
\usepackage[mathcal]{eucal}

\usepackage[left=3cm,right=3cm,
    top=3cm,bottom=3cm,bindingoffset=0cm]{geometry}

\newtheorem{theorem}{{\sc Theorem}}[section]

\newtheorem{lemma}{{\sc Lemma}}[section]
\newtheorem{cor}{{\sc Corollary}}[section]

\newtheorem{definition}{{\sc Definition}}[section]

\begin{document}

\author{Tigran Hakobyan}
\title{\textbf {On the $P_1$ property of sequences of positive integers}} 
 \date{}
\maketitle

\begin{flushleft}
\textbf{\ \ \  Abstract.} \ It is well-known that for any non-constant polynomial $P$ with integer coefficients the sequence $(P(n))_{ n\in \mathbb N}$ has the property that there are infinitely many prime numbers dividing at least one term of this sequence. Certainly, there is a proof based on the Chinese Remainder Theorem. In this paper we give proofs of two analytic criteria revealing this property of sequences.
\end{flushleft}
\subsection{Introduction}
\begin{definition}
We say that a sequence $(n_k)_{k\in \mathbb N}$ of positive integers has the $P_1$ property (we write $(n_k)_{k\in \mathbb N}\in P_1 $) if there exist infinitely many prime numbers dividing at least one term of this sequence. The main results of this paper are the following two theorems.
\end{definition}
\begin{theorem}
If $(n_k)_{k\in \mathbb N}$ is an increasing sequence of positive integers and $$\liminf_{k\rightarrow\infty}{ \frac{\ln(\ln(n_k))}{\ln(k)}}=0$$ then $(n_k)_{k\in \mathbb N}\in P_1$. 
\end{theorem}

\begin{theorem}
 If $(n_k)_{k\in \mathbb N}$ and $(m_k)_{k\in \mathbb N}$ are increasing sequences of positive integers  such that $gcd(n_k,n_{k+l})<m_l$ for all positive integers $k$ and $l$ then $(n_k)_{k\in \mathbb N}\in P_1$.
\end{theorem}
\subsection{Proof of theorem 1}
Suppose we are given positive numbers $w_1, w_2,...,w_n$ where $n\in\mathbb N$. 
\begin{definition} For any $W>0$ define $N(W;w_1,w_2,...,w_n)=$card$\{(k_1, k_2,...,k_n)|k_i\geq 0$, $1\leq i\leq n$, $\sum_{i=1}^nk_iw_i\leq W\}.$
\end{definition}
We will use the inequality $$N(W;w_1,w_2,...,w_n)\leq \frac{(W+\sum_{i=1}^n w_i)^n}{n!\prod_{i=1}^n w_i}$$  mentioned in [5].
\begin{proof}
Now suppose $(n_k)_{k\in \mathbb N}\notin P_1$. So there is a finite set $S=\{p_1,p_2,...,p_n\}$ consisting of prime numbers such that each term of $(n_k)_{k\in \mathbb N}$ is a product of some, not necessary distinct elements from $S$. 
\begin{definition}
For each $l\in N$ let us define $t_l=$card$\{k|1\leq k\leq n, n_k\leq l\}$.
\end{definition}
Hence $t_l\leq$card$\{(k_1, k_2,...,k_n)|k_i\geq 0$, $1\leq i\leq n$, $\prod_{i=1}^n p_i^{k_i}\leq l\}$ which is equivalent to $t_l\leq$card$\{(k_1, k_2,...,k_n)|k_i\geq 0$, $1\leq i\leq n$, $\sum_{i=1}^n \ln(p_i)k_i\leq \ln(l)\}$ where the latter number is $N(\ln(l),\ln(p_1),\ln(p_2),...,\ln(p_n))$ due to definition 2. \newline
Since $$N(W;w_1,w_2,...,w_n)\leq \frac{(W+\sum_{i=1}^n w_i)^n}{n!\prod_{i=1}^n w_i}$$ we therefore have that there is $c>0$ such that $$N(W;w_1,w_2,...,w_n)\leq cW^n$$ for all $W>\delta>0$.  
Consequently for some $a>0$ $$t_l\leq N(\ln(l),\ln(p_1),\ln(p_2),...,\ln(p_n))\leq a(\ln(l))^n$$ for all $l\in N$, $l\geq 2$.\newline 
If we substitute $l=n_k$ for $k=2,3,...$ we will get that $$k=t_{n_k}\leq a(\ln(n_k))^n$$ hence $$\ln(k)\leq \ln(a)+n\ln(\ln(n_k)),  \ k=2,3,...$$ \newline
Thereby $$\liminf_{k\rightarrow\infty}{ \frac{\ln(\ln(n_k))}{\ln(k)}}\geq 1/n>0$$ which is a contradiction.\newline
So, $(n_k)_{k\in \mathbb N}\in P_1$ and the theorem is proved.
\end{proof}
\begin{cor}
 For any non-constant polynomial $P$ with integer coefficients the sequence $(P(n))_{n\in \mathbb N}\in P_1$.
\end{cor}
\begin{proof}
The sequence $(P(n))_{n\in \mathbb N}$ is eventually monotone and
$$\lim_{k\rightarrow\infty}{ \frac{\ln(\ln(P(k)))}{\ln(k)}}=0.$$ It remains to use theorem 1.

\end{proof}

\subsection{Proof of theorem 2}
\begin{proof}
Suppose $(n_k)_{k\in \mathbb N}\notin P_1$. So there is a finite set $S=\{p_1,p_2,...,p_s\}$ consisting of prime numbers $p_1<p_2<...<p_s$ such that each term of $(n_k)_{k\in \mathbb N}$ is a product of some, not necessary distinct elements from $S$. Since $(n_k)_{k\in \mathbb N}$ is increasing it is unbounded hence there is at least one $p\in S$ such that $(\nu_p(n_k))_{k\in \mathbb N}$ is unbounded, where 
$\nu_{p}(m)=\max\{k:p^k|m\}$ for any integer $m$ and prime number $p$. WLOG we may assume that the set of such primes $p$ is $\{p_1,p_2,...,p_l\}$, for some $1\leq l\leq s$.
\begin{definition}
For each $1\leq t\leq l$ and $M\in\mathbb N$ we define $$A_t(M)=\{k|\nu_{p_t}(n_{k})>M\}=(s_{t,j})_{j\in \mathbb N}$$.
\end{definition}
\begin{cor}
$$\mathbb N=\bigcup_{t=1}^l A_t(M)\cup A_M,$$ for some finite set $A_M$. Moreover, $A_M=\{k|\nu_{p_t}(n_k)\leq M, t=1,2,...,l\}$.
\end{cor}
Let us choose $M$ large enough to satisfy $2^M>m_l.$
\begin{lemma}
${s_{t,j+1}}-{s_{t,j}}>l$ for all $t\in\{1,2,...,l\}$ and $j\in \mathbb N$.
\end{lemma}
\begin{proof}
One has that $p_t^M|n_{s_{t,j+1}}$ and $p_t^M|n_{s_{t,j}}$, so $$m_{({s_{t,j+1}}-{s_{t,j}})}>gcd(n_{s_{t,j+1}},n_{s_{t,j}})\geq {p_t}^M\geq 2^M>m_l$$ hence ${s_{t,j+1}}-{s_{t,j}}>l$ as desired.
\end{proof}
Therefore, for any $t\in\{1,2,...,l\}$ and $N\in \mathbb N$ there are at most $([\frac{N}{l+1}]+1)$ elements of $A_t(M)$ in $\{1,2,...,N\}$. Hence for each $N$ there are at least $$\Delta(N)=N-l([\frac{N}{l+1}]+1)>\frac{N}{l+1}-l$$ elements of $\{1,2,...,N\}$ which are not in $\bigcup_{t=1}^l A_t(M)$. Now notice that $\Delta(N)\rightarrow\infty$, hence $A_M$ is infinite, which is a contradiction. \newline 
So, $(n_k)_{k\in \mathbb N}\in P_1$ and the theorem is now proved.
\end{proof}


\newpage

		
\begin{thebibliography}{99}





\bibitem{1}K.Chandrasekharan, \textit{Introduction to analytic number theory}, Springer, (1968).

\bibitem{1}I.M.Vinogradov, \textit{Elements of number theory}, Dover Publ., New York, (2003).

\bibitem{3}A. A. Buhshtab, \textit{Theory of Numbers} (in russian), Moscow, (1974),

\bibitem{4}E.Ram Murty, \textit{Problems in analytical number theory}, Springer, 1998.

\bibitem{5}Padberg, Manfred W.. 1971. "A Remark on``An Inequality for the Number of Lattice Points in a Simplex''. SIAM Journal on Applied Mathematics 20 (4)

\end{thebibliography}
\end{document}